\documentclass[11pt, reqno, oneside, notitlepage]{amsart}

\usepackage[a4paper, total={6in, 9.1in}]{geometry}

\usepackage{amsmath,amscd}
\usepackage{amssymb}
\usepackage{amsthm}
\usepackage{comment}
\usepackage{graphicx}
\usepackage{epstopdf} 
\usepackage{mathrsfs}
\usepackage{cite}

\usepackage{bm}
\usepackage{hyperref}
\usepackage{xcolor}
\hypersetup{
	colorlinks,
	linkcolor={blue},
	urlcolor={blue},
	citecolor={red}
}

\theoremstyle{plain}
\newtheorem{thm}{Theorem}[section]
\newtheorem{prop}{Proposition}[section]

\newtheorem{lem}[prop]{Lemma}

\newtheorem{defi}[prop]{Definition}
\newtheorem{rmk}[prop]{Remark}

\newtheorem*{proposition*}{Proposition}

\numberwithin{equation}{section}

\newcommand {\p} {\partial}

\def\div{\text{div}}

\title[Simultaneous recovery in Mean Field Games system ]{Simultaneously recovering running cost and Hamiltonian in Mean Field Games system}

\author[H. Liu]{Hongyu Liu}
\address{Department of Mathematics, City University of Hong Kong, Kowloon, Hong Kong SAR, China}
\email{hongyu.liuip@gmail.com, hongyliu@cityu.edu.hk}

\author[S. Zhang]{Shen Zhang}
\address{Department of Mathematics, City University of Hong Kong, Kowloon, Hong Kong SAR, China}
\email{szhang347-c@my.cityu.edu.hk}

\begin{document}
	\maketitle
	
	\begin{abstract}
	We propose and study several inverse problems for the mean field games (MFG) system in a bounded domain. Our focus is on simultaneously recovering the running cost and the Hamiltonian within the MFG system by the associated boundary observation. There are several technical novelties that make the study intriguing and challenging. First, the MFG system couples two nonlinear parabolic PDEs with one moving forward and the other one moving backward in time. Second, there is a probability density constraint on the population distribution of the agents. Third, the simultaneous recovery of two coupling factors within the MFG system is technically far from being trivial. Fourth, we consider both cases that the running cost depends on the population density locally and non-locally, and the two cases present different technical challenges for the inverse problem study. We develop two mathematical strategies that can ensure the probability constraint as well as effectively tackle the inverse problems, which are respectively termed as high-order variation and successive linearisation. In particular, the high-order variation method is new to the literature, which demonstrates a novel concept to examine the inverse problems by non-negative inputs only. We believe the methods developed can find applications to inverse problems in other contexts. 
	
\medskip


	
%
	\end{abstract}

	\tableofcontents
	\section{Introduction}
	
\subsection{Mathematical setup}

Initially focusing on the mathematics, but not the physical and practical context, we present the MFG (Mean Field Game) system for our study. Let $\mathbb{R}^n$ be the Euclidean space with $n\in\mathbb{N}$. It is emphasised that $n$ can be any integer greater than or equal to 1. Let $x\in\mathbb{R}^n$ denote the state variable and $t\in [0, \infty)$ be the time variable. Let $\Omega\subset\mathbb{R}^n$ be a bounded Lipschitz domain and $\nu$ be the exterior unit normal to $\partial\Omega$. 
We consider the following MFG system: 
	\begin{equation}\label{main}
		\left\{
		\begin{array}{ll}
			\displaystyle{-\partial_t u(x,t) -\beta\Delta u(t,x)+H(x, m, \nabla u)= F(x,m)} &  {\rm{in}}\ \ Q,\medskip\\
			\displaystyle{\partial_tm(x,t)-\beta\Delta m(x,t)-{\rm div} \big(m(x,t) \nabla_p H(x, m, \nabla u) \big)=0} & {\rm{in}}\ \ Q,\medskip\\
			\p_{\nu} m(x,t)=0,\quad mH_p(x,m,\nabla u)\cdot\nu(x)=0 & {\rm{on}}\ \ \Sigma,\medskip\\
			u(x,T)=\psi(x),\ m(x,0)=m_0(x) & {\rm{in}}\ \ \Omega,\medskip
		\end{array}
		\right.
	\end{equation}
	where $T\in\mathbb{R}_+$ signifies the terminal time, $Q:=\overline{\Omega}\times[0,T]$, and $\Sigma:=\p\Omega\times[0,T]$. In \eqref{main}, $\beta$ is a non-negative constant; $u(x, t)$, $m(x, t)$, $\psi(x), m_0(x)$ are all real-valued functions; and $F(x, m): \mathbb{R}\times\mathbb{R}\mapsto\mathbb{R}$ and $H(x, m, p):\mathbb{R}\times\mathbb{R}\times\mathbb{R}^n\mapsto\mathbb{R}$ with $p:=\nabla u$ are also real-valued functions. Hence, the first equation in \eqref{main} is understood to be backward in time, whereas the second one to be forward in time. Moreover, $m$ in the second equation in \eqref{main} has to be understood to be a probability density in the following sense.

	 Let $\mathcal{P}$ stand for the set of Borel probability measures on $\mathbb{R}^n$, and  $\mathcal{P}(\Omega)$ stand for the set of Borel probability measures on $\Omega$. It is required that $m(\cdot, t)\in \mathcal{P}(\Omega)$ for any given $t\in [0, T]$, and in particular this means that $m_0(\cdot)=m(\cdot, 0)\in\mathcal{P}(\Omega)$. In the function setting, we can define
	\begin{equation}\label{eq:distr1}
		\mathcal{O}:=\{ \mathfrak{m}:\Omega\to [0,\infty) \ \ |\ \ \int_{\Omega} \mathfrak{m}(x)\, dx =a \leq1 \text{ for some } a\in (0,1] \}.
	\end{equation}
The probability density constraint introduced above is equivalent to requiring that $m_0(\cdot)\in \mathcal{O}$ and $m(\cdot, t)\in \mathcal{O}$ for any subsequent $t\leq T$. Indeed, from the second equation in \eqref{main} and using the homogeneous Neumann boundary condition of $m$, one can easily infer that 
\begin{equation}\label{eq:fact1}
\frac{d}{dt}\int_{\Omega} m(x, t)\, dx=0, 
\end{equation}
and hence 
\begin{equation}\label{eq:fact2}
\int_{\Omega} m(x, t)\, dx=\int_{\Omega} m_0(x)\, dx\quad\mbox{for any $t\leq T$}. 
\end{equation}
Nevertheless, in addition to the fact in \eqref{eq:fact1} and \eqref{eq:fact2}, the probability density constraint requires that $m(x,t)$ should be non-negative for all $(x, t)\in Q$, and in particular $m_0(x)$ is nonnegative for all $x\in\overline{\Omega}$. It is also noted that in general, one might want to require $a=1$ in \eqref{eq:distr1}, and we refer to \cite{LMZ} for related discussion on the practical motivation of relaxing the requirement to be $0<a\leq 1$. In fact, we shall provide more discussion on this aspect in what follows. 

We shall discuss the well-posedness of the MFG system \eqref{main} in what follows, especially in Section~\ref{section wp} for certain specific setups of our inverse problem study. Now, assuming the well-posedness of the MFG system \eqref{main}, we introduce the following boundary measurement map:
\begin{equation}\label{eq:meop0}
		\mathcal{N}_{F, H}(m_0,\psi)=(u(x,t)\Big|_{\Sigma}, u(x,0)), 
	\end{equation}
	where
	$(u, m)$ is the (unique) pair of solutions to the MFG system \eqref{main} associated with the initial population distribution $m(x, 0)=m_0(x)$ and the terminal cost $u(x,T)=\psi(x)$. That is, $\mathcal{N}_{F,H}$ sends the pair of inputs consisting of the initial population and terminal cost to the observation dataset $(u|_\Sigma, u(x, 0))$. Clearly, $\mathcal{N}_{F, H}$ encodes all the lateral Cauchy data as well as the initial data of $u$ in \eqref{main}. $u(x,0)$ signifies the so-called total cost of the mean field game described by \eqref{main}; see our discussion in what follows. In this article, we are mainly concerned with the following inverse problem:
	\begin{equation}\label{eq:ip1}
		\mathcal{N}_{F,H}\rightarrow (F, H). 
	\end{equation}	
That is, by knowing the boundary measurement operator $\mathcal{N}_{F, H}$ associated with the MFG system \eqref{main} with certain unknown $F$ and $H$, one intends to recover the unknowns. For the inverse problem \eqref{eq:ip1}, we shall mainly study the unique identifiability issue by establishing the following one-to-one correspondence:
\begin{equation}\label{eq:issue1}
\mathcal{N}_{F_1, H_1}=\mathcal{N}_{F_2, H_2}\quad\mbox{if and only if}\quad (F_1, H_1)\equiv (F_2, H_2), 
\end{equation}
where $(F_j, H_j)$, $j=1, 2$, are two sets of target configurations.

\subsection{Background and motivation}

Mean Field Games (MFGs) are non-atomic differential games which offer quantitative modelling of the macroscopic behaviours of a large number of symmetric agents seeking to minimise a specific cost. Caines-Huang-Malhame \cite{HCM06} and Lasry-Lions \cite{Lions} independently pioneered the MFG theory in 2006, and it has since then received significant attention and increasing studies in the literature. One of the main characteristics of an MFG is the existence of an adversarial regime, and the Nash equilibrium exists and is unique in this so-called monotone regime. In general, we focus on the case that the function $H(x,m,p)$ is convex with respect to the third variable $p$. This condition implies that the first equation in $\eqref{main}$ is associated with
an optimal control problem. $u$ is the value function associated with a typical small player and $m$ is the population density, which respectively fulfill the Hamilton-Jacobi-Bellman (HBJ) and Fokker-Planck-Kolmogorov (FPK) equations, i.e. the first and second equations in \eqref{main}. These two equations are coupled nonlinearly in a manner with the time evolution forward in one equation and backward in the other one.
In such a non-atomic differential game, an average agent can control the following stochastic differential equation
\begin{equation}
	d X_t=\alpha_tdt+\sqrt{2\beta}dB_t,
\end{equation}
 where $B_t$ is a standard Brownian motion and $\alpha_t$ belongs to the set of admissible controls. One may choose it to be the set of all functions that are Borel measurable, and uniformly Lipschitz continuous in the space variable. The corresponding optimal control problem is that the agent aims to minimizing the following  expectation
\begin{equation}
	J(x;m_t,\alpha)=\mathbb{E}\Big[\int_0^T L(X_s,m_s,\alpha_s) + F(X_s,m_s) ds+ G(X_t,m_T)\Big],
\end{equation}
where $L$ is the Fenchel conjugate of $H$ with respect to the third variable. If the equilibrium exists, the agent's optimal control is given by $\alpha^*=-D_pH(x,m,\nabla u)$. Since we consider non-atomic differential games, all agents argue in this way in the equilibrium. This fact leads the diffusion of agents under the drift term $-D_pH(x,m,\nabla u)$. It is described by the  FPK equation.

 In the physical setup, $H$ is the Hamiltonian of the MFG system and $F$ is the running cost that characterises the gaming strategies/rules among the players. $m_0$ is the initial population and $\psi$ is the terminal cost. As mentioned earlier, $u(x, 0)$ signifies the total cost which can be known at the end of the game. The constant $\beta$ signifies the volatility corresponding to the idiosyncratic noise of the small players. Here, we also note that the homogeneous Neumann boundary conditions, namely $\partial_\nu u=\partial_\nu m=0$ on $\Sigma$, signify that the player cannot leave the domain $\Omega$, and will be reflected back to the domain when meeting the boundary. Here, we would like to remark one practical rationale behind our current study that in \eqref{eq:distr1}, we shall not require that $a=1$, though $m$ is treated as a probability density. In fact, one can consider the scenario that the MFG domain consists of a family of disjoint subdomains, say $D_j$, $j\in\mathbb{N}$, such that the overall population on $\cup_j D_j$ is 1, namely $\int_{\cup_j D_j} m=1$. Though those subdomains are disjoint, the agents within each subdomain can interact with those in other subdomains, say e.g. via internet. Hence if $\Omega$ is taken to be any one of those subdomains, i.e. $D_j$, it is not necessary to require that $\int_\Omega m=1$. That is, $a$ in \eqref{eq:distr1} can be any number in $(0, 1]$, as long as $m$ is required to be nonnegative. We also refer to \cite{LMZ} for more related discussion on the practical motivation of relaxing this requirement. Such a relaxation enables us to simultaneously recover the Hamiltonian $H$ and the running cost $F$, which we shall discuss again in what follows.  
 
 The well-posedness of the MFG system \eqref{main} is well understood in different contexts and still remains an active subject of many ongoing studies. The first results date back to the original works of Lasry and Lions and have been presented in Lions \cite{Lions} and see also Caines-Huang-Malhame \cite{HCM06}. One can take into account both non-local and local dependences on the measure $m$ for $F$ and $G$. In the case of non-local data $F$ and $G$, its well-posedness is known in \cite{Cardaliaguet,CarPor,CarDel-I}. In the case that $F,G$ are locally dependent on the measure variable $m$, we refer to\cite{Amb:18,Car,CarGra,CirGiaMan,GomPimSan:15,Por,FerGomTad}.
  We call this the forward problem in our study. On the other hand, the inverse problems for MFGs are far less studied in the literature. To our best knowledge, there are only several numerical results available in \cite{CFLNO,DOY}. In \cite{LMZ}, the authors investigated inverse problems for an MFG system with unknown running cost and total cost, and established the unique identifiability results. It is pointed out in \cite{LMZ} that the probability density constraint on $m$ \eqref{eq:distr1} significantly increases the difficulty of the inverse problem study, since one would need to construct suitable ``probing modes" which fulfil this constraint. In \cite{MFG2}, to overcome this difficulty, the authors proposed an effective method for dealing with the inverse problem, which takes into account the high-order linearisation in the probability space around a nontrivial uniform distribution. Moreover, this method can be applied to the situation where the running cost varies with respect to $t$ by the construction of certain so-called CGO (Complex-Geometric-Optics) solutions. However, the method in \cite{MFG2} cannot be extended to dealing with the case of recovering more than one unknowns within the MFG system, say e.g. both the running cost and the Hamiltonian considered in the current article. In the next subsection, we shall provide more related discussions. 
 
 In Table~\ref{table:1}, we provide a brief summary of different studies in the literature on MFG inverse problems which make use of measurements associated with infinitely many events. We would also like to mention the recent works \cite{ILY,ILY2,carleman1,KlAv,KLL1,KLL2,LY} on MFGs from the perspective of inverse problems associated with a single-event measurement. Due to the reduction of the data usage, one can infer less information within the MFG system compared to those studies listed in Table~\ref{table:1}. 
 
\begin{table}[htp]
\caption{Summary of different studies in the literature on MFG inverse problems}
\begin{center}
\begin{tabular}{|p{2cm}|p{4cm}|p{4cm}|p{4cm}|}
     	\hline Literature  & Unknowns for Recovery & Features of the Unknowns & Mathematical Strategies and Restrictions\\
     	\hline   \cite{LMZ} & The running cost $F$ and the terminal cost $\psi$.  & $F$ and $\psi$ belong to certain analytic classes. & The probability measure constraint was not treated; the argument is based on first-order variation and high-order linearisation.\\
     	\hline   \cite{MFG2} & The running cost $F$. & $F$ belongs to an analytic class. & The argument is based on first-order variation and high-order linearisation around a nontrivial uniform distribution. It is difficult to recover more than one unknowns within the MFG system. \\ 
     	\hline    \cite{DOY} & The running cost $F$ and the Hamiltonian $H$. & $F$ is in the form of an integral, namely $F$ depends on $m$ non-locally. $H$ is kinetic. & Numerical results only with no theoretical justification. \\ 
     	\hline   The current article & The running cost $F$ and the Hamiltonian & $F$ is in an analytic class or in the form of an integral. Hence, $F$ may depend on $m$ locally and non-locally. $H$ is kinetic. & The argument is based on high-order variation and high-order linearisation. The probability density constraint is treated but with a certain relaxation. One can recover both $F$ and $H$ simultaneously.  \\
     	\hline
     \end{tabular}
\end{center}
\label{table:1}
\end{table}%

\subsection{Discussion on technical developments}	

According to our earlier discussion, there are several technical novelties that makes the study of the inverse problem \eqref{eq:ip1}--\eqref{eq:issue1} associated with the MFG system \eqref{main} highly intriguing and challenging, at least from a mathematical viewpoint. First, in the forward system \eqref{main}, the HBJ and FPK equations are coupled nonlinearly and backward-forward in time. In order to tackle the nonlinearity, a powerful strategy is based on the linearisation. In fact, the so-called successive/high-order linearisation technique has been extensively developed for a variety of inverse problems associated with nonlinear PDEs; see e.g. \cite{KLU,LLLZ} and the references cited therein. Most of those studies are concerned with a single-type of nonlinear PDE, and in \cite{LMZ} the method was developed for a MFG system which couples two nonlinear PDEs. It is pointed out that the period property of the MFG system considered in \cite{LMZ} contributes to resolving the probability density constraint. In \cite{MFG2}, the successive linearisation technique was further developed by taking into account the probability density constraint. Among several technical developments, one of the key ideas in \cite{MFG2} is to do the linearisation around a uniform distribution of the following form
\begin{equation}\label{eq:sl1}
m=\frac{1}{|\Omega|}+\delta_m\quad\mbox{with}\ \ \int_\Omega\delta_m=0. 
\end{equation}
Though the linearisation technique works to ensure the probability density constraint and at the same time to tackle the inverse problem, the resulting linearised MFG system is still a coupled one and moreover there is an average-zero constraint on the first-order variation $\delta_m$. Hence, only the running cost $F$ is recovered in \cite{MFG2} and moreover, it is required that $F$ depends locally on $m$. Nevertheless, we also note that in \cite{MFG2}, in addition to $x$ and $m$, $F$ can depend explicitly on $t$, i.e. the running cost can be of function of the form $F(x, t, m)$, whereas in the current article, we have to assume that the running cost is a function of the form $F(x, m)$. Hence, in the context of MFG inverse problems, two natural questions arise. Can we find a better way to describe the probability measure constraint on $m$? Is it possible that we can understand the space spanned by positive solutions (probability measure is non-negative at least) of a specific PDE system?

In this paper, we develop a novel approach to ensure the probability measure constraint on $m$ while effectively tackle the MFG inverse problems.
We term this approach as high-order variation in combination with the successive/high-order linearisation. This approach is particularly powerful in tackling the case that $F$ depends on $m$ non-locally. Next, we briefly explain the benefits of this strategy since it offers a novel concept to examine the inverse problems by nonnegative inputs only, which is a topic that is rarely touched in the mathematical theory of inverse problems. In fact, by allowing the input to have high-order variations with respect to an asymptotic parameter, 
this approach allows us to primarily concentrate on the second-order linearised system but not the first-order one.  Moreover, the input of the first-order linearisation system is forced to be nonnegative, and hence the input of the second-order linearization system can be arbitrary. In doing so, we can ensure the probability measure constraint, and moreover we can consider the linearization system of the MFG system around a pair of trivial solutions, which is the main reason for us to relax the requirement \eqref{eq:distr1} from $a=1$ to $a\in [0, 1]$. This allows us to simultaneously recover the Hamiltonian and the running cost, and moreover the running cost can depend on $m$ locally or non-locally. 

The two questions that we raised before are not possessed exclusively by the MFG inverse problems. In fact, in a recent article \cite{Lo}, a list of coupled or non-coupled PDEs was discussed where the probability density constraint naturally occurs. Inverse boundary problems associated with those PDEs are proposed and studied, where a high-order variation and high-order/successive linearisation scheme was also developed. We believe  the mathematical strategies and techniques developed in the current article as well as those in \cite{Lo} offer novel perspectives on inverse boundary problems in those new and intriguing contexts which have high potential to produce more results of both theoretical and practical importance. 

The rest of the paper is organized as follows. In Section 2, we collect some preliminary results and present the main results of the inverse problem. Section 3 is devoted to the study of the forward MFG system. In Section 4, we develop the high-order variation and high-order/successive linearisation methods, and in Section 5, we give the proofs of the main theorems. In Section 6, we compare and discuss the high-order variation method and the high-order linearisation method. In particular, it shows the necessity of high-order variation method to ensure the probability density constraint. 

	\section{Preliminaries and statement of main results}
	For $k\in\mathbb{N}$ and $0<\alpha<1$, the H\"older space $C^{k+\alpha}(\overline{\Omega})$ is defined as the subspace of $C^{k}(\overline{\Omega})$ such that $\phi\in C^{k+\alpha}(\overline{\Omega})$ if and only if $D^l\phi$ exist and are H\"older continuous with exponent $\alpha$ for all $l=(l_1,l_2,\ldots,l_n)\in \mathbb{N}^n$ with $|l|\leq k$, where $D^l:=\partial_{x_1}^{l_1}\partial_{x_2}^{l_2}\cdots\partial_{x_n}^{l_n}$ for $x=(x_1, x_2,\ldots, x_n)$. The norm is defined as
	\begin{equation}
		\|\phi\|_{C^{k+\alpha}(\overline{\Omega}) }:=\sum_{|l|\leq k}\|D^l\phi\|_{\infty}+\sum_{|l|=k}\sup_{x\neq y}\frac{|D^l\phi(x)-D^l\phi(y)|}{|x-y|^{\alpha}}.
	\end{equation}
	If the function $\phi$ depends on both the time and space variables, we define $\phi\in C^{k+\alpha, \frac{k+\alpha}{2}}(Q)$ if $D^lD^{j}_t\phi$ exist and are H\"older continuous with exponent $\alpha$ in $x$ and $\frac{k+\alpha}{2} $ in $t$ for all  $l\in \mathbb{N}^n$, $j\in\mathbb{N}$ with $|l|+2j\leq k.$ The norm is defined as
	\begin{equation}
		\begin{aligned}
			\|\phi\|_{ C^{k+\alpha, \frac{k+\alpha}{2}}(Q)}:&=\sum_{|l|+2j\leq k}\|D^lD^j_t\phi\|_{\infty}+\sum_{|l|+2j= k}\sup_{t,x\neq y}\frac{|\phi(x,t)-\phi(y,t)|}{|x-y|^{\alpha}}\\
			&+\sum_{|l|+2j= k}\sup_{t\neq t',x} \frac{|\phi(x,t)-\phi(x,t')|}{|t-t'|^{\alpha/2}}.
		\end{aligned}
	\end{equation}
	
\subsection{Admissible class}\label{assumption}
Now we introduce the admissible classes of the running costs $F$ in two different cases. The first one is close to the conditions in \cite{LMZ}. For the completeness of this paper, we list it here.
\begin{defi}\label{Admissible class2}
	We say $U(x,z):\mathbb{R}^n\times\mathbb{C}\to\mathbb{C}$ is admissible, denoted by $U\in\mathcal{A}$, if it satisfies the following conditions:
	\begin{enumerate}
		\item[(i)] The map $z\mapsto U(\cdot,z)$ is holomorphic with value in $C^{\alpha}(\mathbb{R}^n)$ for some $\alpha\in(0,1)$;
		\item[(ii)] $U(x,0)=0$ for all $x\in\mathbb{R}^n$. 
	\end{enumerate} 	
	Clearly, if (1) and (2) are fulfilled, then $U$ can be expanded into a power series as follows:
	\begin{equation}\label{eq:G}
		U(x,z)=\sum_{k=1}^{\infty} U^{(k)}(x)\frac{(z)^k}{k!},
	\end{equation}
	where $ U^{(k)}(x)=\frac{\p^kU}{\p z^k}(x,0)\in C^{\alpha}(\mathbb{R}^n).$
\end{defi}
Clearly, if $F(x,m)\in\mathcal{A}$ and $m$ is the density of the measure, then $F$ depends on the measure locally. 

Next, we consider the non-local case.
\begin{defi}\label{admi 2}
	Let $m(x,t)$ be the density of a given distribution. We say $$F(x,m)=\int_{\Omega} K(x,y)m(y,t)dy$$ belongs to $\mathcal{B}$ if 
		\begin{enumerate}
		\item[(i)] $K(x,y)$ is smooth in $\Omega\times\Omega$;
		\item[(ii)] $\int_{\Omega} K(x,y)dy=0 $ for all $x\in\Omega$. 
	\end{enumerate} 	
\end{defi}
\begin{rmk}
	In fact, the condition (ii) in the Definition $\ref{admi 2}$ is quite natural from the MFG point of view. If $m_0(x)=1$ and $\psi(x)$ is a constant $c$, the solution of MFG system should be $(u,m)=(c,1)$. This is because MFGs are non-atomic differential games. In other words, if $m_0(x)$ is the uniform distribution and $\psi(x)$ is constant, it is already an equilibrium of the system, and hence $m(x,t)$ should keep to be $1$. This is a common nature of the MFG system that the uniform distribution is a stable state. 
\end{rmk}

Throughout the rest of the paper, we assume that the Hamiltonian in \eqref{main} is of the following form:
\begin{equation}\label{eq:ham1}
H(x,p)=\frac 1 2\kappa(x)|p|^2, \quad p=\nabla u(x, t), 
\end{equation}
which signifies that the Lagrangian energy of the MFG system is kinetic (cf. \cite{LMZ}). Here, $\kappa(x)$ is usually assumed to be a non-negative-valued function, but we shall not need this assumption as long as it is a real-valued function. In this case, the boundary condition is just $\p_{\nu} u(x,t)=\p_{\nu}m(x,t)=0$. This type of Hamiltonian widely occurs in the MFG theory (cf. \cite{DOY} ). Hence, in \eqref{eq:ip1}, recovering $H$ is equivalently to recovering the function $\kappa(x)$.

\subsection{Main unique identifiability results}
We are able to articulate the primary conclusions for the inverse problems, which show that
one can recover  the running cost and Hamiltonian from the measurement map $\mathcal{N}_{F,H}$. 

\begin{thm}\label{der F}
       Assume that  $F_j(x,m)\in\mathcal{A} $. Let $\mathcal{N}_{F_j,H_j}$, $j=1,2$, be the measurement map associated to
	the following system:
	\begin{equation}\label{eq:mfg1}
		\begin{cases}
			-\p_tu(x,t)-\beta\Delta u(x,t)+\frac 1 2 \kappa_j{|\nabla u(x,t)|^2}= F_j(x,m),& \text{ in }  Q,\medskip\\
			\p_t m(x,t)-\beta\Delta m(x,t)-{\rm div} (m(x,t) \kappa_j\nabla u(x,t))=0,&\text{ in } Q,\medskip\\
			\p_{\nu} u(x,t)=\p_{\nu}m(x,t)=0      &\text{ on } \Sigma,\medskip\\
			u(x,T)=\psi(x), & \text{ in } \Omega,\medskip\\
			m(x,0)=m_0(x), & \text{ in } \Omega.\\
		\end{cases}  		
	\end{equation}	
	If for any $(m_0,\psi)\in [ C^{2+\alpha}(\Omega) \cap \mathcal{O}]\times C^{2+\alpha}(\Omega) $, where $\mathcal{O}$ is defined in \eqref{eq:distr1}, one has 
	$$\mathcal{N}_{F_1,H_1}(m_0,\psi)=\mathcal{N}_{F_2,H_2}(m_0,\psi),$$    then it holds that 
	$$\kappa_1=\kappa_2\ \text{  in  }\ \Omega,$$ and 
	$$F_1(x,z)=F_2(x,z)\ \text{  in  }\ \Omega\times \mathbb{R}.$$ 
\end{thm}
\begin{thm}\label{der F2}
	Assume that  
	$$F_j(x,m)=\int_{\Omega}K_j(x,y)m(y,t)dy\in\mathcal{B},\ j=1,2. $$ 
	Let $\mathcal{N}_{F_j,H_j}$, $j=1,2$, be the measurement map associated to
	the following system:
	\begin{equation}\label{eq:mfg2}
		\begin{cases}
			-\p_tu(x,t)-\beta\Delta u(x,t)+\frac 1 2 \kappa_j{|\nabla u(x,t)|^2}= \int_{\Omega}K_j(x,y)m(y,t)dy,& \text{ in }  Q,\medskip\\
			\p_t m(x,t)-\beta\Delta m(x,t)-{\rm div} (m(x,t) \kappa_j\nabla u(x,t))=0,&\text{ in } Q,\medskip\\
			\p_{\nu} u(x,t)=\p_{\nu}m(x,t)=0      &\text{ on } \Sigma,\medskip\\
			u(x,T)=\psi(x), & \text{ in } \Omega,\medskip\\
			m(x,0)=m_0(x), & \text{ in } \Omega.\\
		\end{cases}  		
	\end{equation}	
	If for any $(m_0,\psi)\in [ C^{2+\alpha}(\Omega) \cap \mathcal{O}]\times C^{2+\alpha}(\Omega) $, where $\mathcal{O}$ is defined in \eqref{eq:distr1}, one has 
	$$\mathcal{N}_{F_1,H_1}(m_0,\psi)=\mathcal{N}_{F_2,H_2}(m_0,\psi),$$    then it holds that 
	$$\kappa_1=\kappa_2\ \text{  in  }\ \Omega,$$ and 
	$$K_1(x,y)=K_2(x,y)\ \text{  in  }\ \Omega\times \Omega.$$ 
\end{thm}


\section{Local well-posedness of the forward problems }\label{section wp}

We prove several auxiliary results on the forward problem of the MFG system $\eqref{main}$ in this section. One of the most significant insights is the infinite differentiability of the system with respect to small variations around given input $m_0(x)$ or $\psi(x)$. This forms the basis to develop the linearization method in what follows for the associated inverse problems.


In the next two theorems, we show the (infinite) differentiability of the MFG system \eqref{main} with respect to small variations of $\psi(x)$ and $m_0(x)$, respectively.  

\begin{thm}\label{local_wellpose1}
	Given $\psi(x)=0$. Suppose that $F\in\mathcal{A}$ . The following result holds:
	\begin{enumerate}		
		\item[(a)]
		There exist constants $\delta>0$ and $C>0$ such that for any 
		\[
		m_0\in B_{\delta}(C^{2+\alpha}(\Omega)) :=\{m_0\in C^{2+\alpha}(\Omega): \|m_0\|_{C^{2+\alpha}(\Omega)}\leq\delta \},
		\]
		the MFG system $\eqref{main}$ has a solution $(u,m)\in
		[C^{2+\alpha,1+\frac{\alpha}{2}}(Q)]^2$ which satisfies
		\begin{equation}
			\|(u,m)\|_{ C^{2+\alpha,1+\frac{\alpha}{2}}(Q)}:= \|u\|_{C^{2+\alpha,1+\frac{\alpha}{2}}(Q)}+ \|m\|_{C^{2+\alpha,1+\frac{\alpha}{2}}(Q)}\leq C\|m_0\|_{ C^{2+\alpha}(\Omega)}.
		\end{equation}
		Furthermore, the solution $(u,m)$ is unique within the class
		\begin{equation}
			\{ (u,m)\in  C^{2+\alpha,1+\frac{\alpha}{2}}(Q)\times C^{2+\alpha,1+\frac{\alpha}{2}}(Q): \|(u,m)\|_{ C^{2+\alpha,1+\frac{\alpha}{2}}(Q)}\leq C\delta \}.
		\end{equation}		
		\item[(b)] Define a function 
		\[
		S: B_{\delta}(C^{2+\alpha}(\Omega)\to C^{2+\alpha,1+\frac{\alpha}{2}}(Q)\times C^{2+\alpha,1+\frac{\alpha}{2}}(Q)\ \mbox{by $S(m_0):=(u,m)$}. 
		\] 
		where $(u,m)$ is the unique solution to the MFG system \eqref{main}.
		Then for any $m_0\in B_{\delta}(C^{2+\alpha}(\Omega))$, $S$ is holomorphic at $m_0$.
	\end{enumerate}
\end{thm}

\begin{thm}\label{local_wellpose1'}
	Given $m_0(x)=0$. Suppose that $F\in\mathcal{A}$ . The following result holds:
	\begin{enumerate}		
		\item[(a)]
		There exist constants $\delta>0$ and $C>0$ such that for any 
		\[
		\psi(x)\in B_{\delta}(C^{2+\alpha}(\Omega)) :=\{\psi\in C^{2+\alpha}(\Omega): \|\psi(x)\|_{C^{2+\alpha}(\Omega)}\leq\delta \},
		\]
		the MFG system $\eqref{main}$ has a solution $(u,m)\in
		[C^{2+\alpha,1+\frac{\alpha}{2}}(Q)]^2$ which satisfies
		\begin{equation}
			\|(u,m)\|_{ C^{2+\alpha,1+\frac{\alpha}{2}}(Q)}:= \|u\|_{C^{2+\alpha,1+\frac{\alpha}{2}}(Q)}+ \|m\|_{C^{2+\alpha,1+\frac{\alpha}{2}}(Q)}\leq C\|\psi(x)\|_{ C^{2+\alpha}(\Omega)}.
		\end{equation}
		Furthermore, the solution $(u,m)$ is unique within the class
		\begin{equation}
			\{ (u,m)\in  C^{2+\alpha,1+\frac{\alpha}{2}}(Q)\times C^{2+\alpha,1+\frac{\alpha}{2}}(Q): \|(u,m)\|_{ C^{2+\alpha,1+\frac{\alpha}{2}}(Q)}\leq C\delta \}.
		\end{equation}		
		\item[(b)] Define a function 
		\[
		S: B_{\delta}(C^{2+\alpha}(\Omega)\to C^{2+\alpha,1+\frac{\alpha}{2}}(Q)\times C^{2+\alpha,1+\frac{\alpha}{2}}(Q)\ \mbox{by $S(\psi(x)):=(u,m)$}. 
		\] 
		where $(u,m)$ is the unique solution to the MFG system \eqref{main}.
		Then for any $\psi(x)\in B_{\delta}(C^{2+\alpha}(\Omega))$, $S$ is holomorphic at $\psi$.
	\end{enumerate}
\end{thm}

The proofs of Theorems~\ref{local_wellpose1} and \ref{local_wellpose1'} are similar and we only present the one for Theorem~\ref{local_wellpose1'}. 

\begin{proof}[Proof of Theorem $\ref{local_wellpose1'}$]
		Let 
	\begin{align*}
		&Y_1:= \{\psi\in C^{2+\alpha}(\Omega ) :\p_{\nu}\psi=0      \}, \\
		&Y_2:\{(u,m)\in C^{2+\alpha,1+\frac{\alpha}{2}}(Q)\times C^{2+\alpha,1+\frac{\alpha}{2}}(Q):\p_{\nu}m=\p_{\nu}u=0  \text{ in }\Sigma\},\\
		&Y_3:=Y_1\times Y_1\times C^{\alpha,\frac{\alpha}{2}}(Q )\times C^{\alpha,\frac{\alpha}{2}}(Q ),
	\end{align*} 
     and define a map $\mathscr{K}:Y_1\times Y_2 \to Y_3$ by that for any $(\psi,\tilde u,\tilde m)\in Y_1\times Y_2$,
	\begin{align*}
		&
		\mathscr{K}( \psi,\tilde u,\tilde m)(x,t)\\
		:=&\big( \tilde u(x,T)-\psi(x), \tilde m(x,0) , 
		-\p_t\tilde u(x,t)-\beta\Delta \tilde u(x,t)\\ &+\frac{\kappa(x)|\nabla \tilde u(x,t)|^2}{2}- F(x,\tilde m(x,t)), 
		\p_t \tilde m(x,t)-\beta\Delta \tilde m(x,t)-{\rm div}(\tilde m(x,t)\kappa(x)\nabla \tilde u(x,t))  \big) .
	\end{align*}

   We begin by demonstrating that $\mathscr{K} $ is well-defined. As a result of the fact that the H\"older space is an algebra under point-wise multiplication, and hence we have 
   $$\kappa(x)|\nabla u|^2, {\rm div}(m(x,t)\kappa(x)\nabla u(x,t))  \in C^{\alpha,\frac{\alpha}{2}}(Q ).$$
	By the Cauchy integral formula,
	\begin{equation}\label{eq:F1}
		F^{(k)}\leq \frac{k!}{R^k}\sup_{|z|=R}\|F(\cdot,z)\|_{C^{\alpha,\frac{\alpha}{2}}(Q ) },\ \ R>0.
	\end{equation}
	Then there is $L>0$ such that for all $k\in\mathbb{N}$,
	\begin{equation}\label{eq:F2}
		\left\|\frac{F^{(k)}}{k!}m^k\right\|_{C^{\alpha,\frac{\alpha}{2}}(Q )}\leq \frac{L^k}{R^k}\|m\|^k_{C^{\alpha,\frac{\alpha}{2}}(Q )}\sup_{|z|=R}\|F(\cdot,z)\|_{C^{\alpha,\frac{\alpha}{2}}(Q ) }.
	\end{equation}
	By choosing $R\in\mathbb{R}_+$ large enough and by virtue of \eqref{eq:F1} and \eqref{eq:F2}, it can be seen that the series 
	$\sum_{k=1}^{\infty}F^{(k)}(x)\frac{z^k}{k!}$
	converges in $C^{\alpha,\frac{\alpha}{2}}(Q )$ and therefore $F(x,m(x,t))\in  C^{\alpha,\frac{\alpha}{2}}(Q ).$ 
	This implies that $\mathscr{K} $ is well-defined.

	Let us show that $\mathscr{K}$ is holomorphic. Verifying $\mathscr{K}$ being weakly holomorphic is sufficient because it is clearly locally confined. That is we aim to show the map
	$$\lambda\in\mathbb C \mapsto \mathscr{K}((m_0,\tilde u,\tilde m)+\lambda (\bar m_0,\bar u,\bar m))\in Y_3,\quad\text{for any $(\bar m_0,\bar u,\bar m)\in Y_1\times Y_2$}$$
	is holomorphic. In fact, this follows from the fact that the series 
	$\sum_{k=1}^{\infty}F^{(k)}(x)\frac{z^k}{k!}$
	converges in $C^{\alpha,\frac{\alpha}{2}}(Q )$ .

	Note that $  \mathscr{K}(0,0,0)=0$. Let us compute $\nabla_{(\tilde u,\tilde m)} \mathscr{K} (0,0,0)$:
	\begin{equation}\label{Fer diff}
		\begin{aligned}
			\nabla_{(\tilde u,\tilde m)} \mathscr{K}(0,0,0) (u,m) =( u|_{t=T}, m|_{t=0}, 
			-\p_tu(x,t)-\beta\Delta u(x,t)-F^{(1)}m, \p_t m(x,t)-\beta\Delta m(x,t)).
		\end{aligned}			
	\end{equation}
	
	Then
	If $\nabla_{(\tilde u,\tilde m)} \mathscr{K} (0,0,0)=0$, we have 
	$ \tilde m=0$ and then $\tilde u=0$. Therefore, the map is an injection. 
	
	On the other hand,  letting $ (r(x),s(x,t))\in C^{2+\alpha}(\Omega)\times C^{\alpha,\frac{\alpha}{2}}(Q ) $, there exists $a(x,t)\in C^{2+\alpha,1+\frac{\alpha}{2}}(Q)$ such that
	\begin{equation*}
		\begin{cases}
			\p_t a(x,t)-\beta\Delta a(x,t)=s(x,t)  &\text{ in } Q,\medskip\\
			\p_{\nu} a=0 &\text{ in } \Sigma,\medskip\\ 
			a(x,0)=r(x)                       &  \text{ in } \Omega .
		\end{cases}
	\end{equation*}
	Then letting $ (r'(x),s'(x,t))\in C^{2+\alpha}(\Omega)\times C^{\alpha,\frac{\alpha}{2}}(Q ) $, one can show that there exists $ b(x,t)\in C^{2+\alpha,1+\frac{\alpha}{2}}(Q)$ such that
	\begin{equation*}
		\begin{cases}
			-\p_t b(x,t)-\beta\Delta b(x,t)-F^{(1)}(x)a=s'(x,t)  &\text{ in } Q,\medskip\\
			\p_{\nu} b=0 &\text{ in } \Sigma,\medskip\\ 
			b(x,T)=r'(x)                  &  \text{ in } \Omega.
		\end{cases}
	\end{equation*}
	This shows that $	\nabla_{(\tilde u,\tilde m)} \mathscr{K}(0,0,0)$ is also surjective.
	Therefore, $\nabla_{(\tilde u,\tilde m)} \mathscr{K} (0,0,0)$ is a linear isomorphism between $Y_2$ and $Y_3$. Hence, by the implicit function theorem, there exist $\delta>0$ and a unique holomorphic function $S: B_{\delta}(\Omega)\to Y_2$ such that $\mathscr{K}(\psi,S(\psi))=0$ for all $m_0\in B_{\delta}(\Omega) $.
	
	By letting $(u,m)=S(\psi)$, we obtain the unique solution of the MFG system \eqref{main}. Let $ (u_0,v_0)=S(0)$. Then since we assume $m_0(x)=0$, we have $(u_0,v_0)=(0,0)$. (The reason why we keep using $m_0(x)$ instead of $0$ in this proof is that we want to show it can be used to show Theorem $\ref{local_wellpose1}$.) 
	
	Since $S$ is Lipschitz, we know that there exist constants $C>0$ such that 
	\begin{equation*}
			\|(u,m)\|_{ C^{2+\alpha,1+\frac{\alpha}{2}}(Q)^2}\\
			\leq C\|\psi(x)\|_{B_{\delta}(\Omega)} 
	\end{equation*}

	The proof is complete. 
	
\end{proof}

We still need to show the local well-posedness in the case that the running cost $F$ is in the form of an integral. This is stated in the following theorems.

\begin{thm}\label{local_wellpose2}
	Given $\psi(x)=0$. Suppose that $F\in\mathcal{B}$ . The following results hold:
	\begin{enumerate}		
		\item[(a)]
		There exist constants $\delta>0$ and $C>0$ such that for any 
		\[
		m_0\in B_{\delta}(C^{2+\alpha}(\Omega)) :=\{m_0\in C^{2+\alpha}(\Omega): \|m_0\|_{C^{2+\alpha}(\Omega)}\leq\delta \},
		\]
		the MFG system $\eqref{main}$ has a solution $(u,m)\in
		[C^{2+\alpha,1+\frac{\alpha}{2}}(Q)]^2$ which satisfies
		\begin{equation}
			\|(u,m)\|_{ C^{2+\alpha,1+\frac{\alpha}{2}}(Q)}:= \|u\|_{C^{2+\alpha,1+\frac{\alpha}{2}}(Q)}+ \|m\|_{C^{2+\alpha,1+\frac{\alpha}{2}}(Q)}\leq C\|m_0\|_{ C^{2+\alpha}(\Omega)}.
		\end{equation}
		Furthermore, the solution $(u,m)$ is unique within the class
		\begin{equation}
			\{ (u,m)\in  C^{2+\alpha,1+\frac{\alpha}{2}}(Q)\times C^{2+\alpha,1+\frac{\alpha}{2}}(Q): \|(u,m)\|_{ C^{2+\alpha,1+\frac{\alpha}{2}}(Q)}\leq C\delta \}.
		\end{equation}		
		\item[(b)] Define a mapping 
		\[
		S: B_{\delta}(C^{2+\alpha}(\Omega))\to C^{2+\alpha,1+\frac{\alpha}{2}}(Q)\times C^{2+\alpha,1+\frac{\alpha}{2}}(Q)\ \mbox{by $S(m_0):=(u,m)$},
		\] 
		where $(u,m)$ is the unique solution to the MFG system \eqref{main}.
		Then for any $m_0\in B_{\delta}(C^{2+\alpha}(\Omega))$, $S$ is holomorphic at $m_0$.
	\end{enumerate}
\end{thm}
Similarly, we have
\begin{thm}\label{local_wellpose2'}
	Given $m_0(x)=0$. Suppose that $F\in\mathcal{A}$ . The following results hold:
	\begin{enumerate}		
		\item[(a)]
		There exist constants $\delta>0$ and $C>0$ such that for any 
		\[
		\psi(x)\in B_{\delta}(C^{2+\alpha}(\Omega)) :=\{\psi\in C^{2+\alpha}(\Omega): \|\psi(x)\|_{C^{2+\alpha}(\Omega)}\leq\delta \},
		\]
		the MFG system $\eqref{main}$ has a solution $(u,m)\in
		[C^{2+\alpha,1+\frac{\alpha}{2}}(Q)]^2$ which satisfies
		\begin{equation}
			\|(u,m)\|_{ C^{2+\alpha,1+\frac{\alpha}{2}}(Q)}:= \|u\|_{C^{2+\alpha,1+\frac{\alpha}{2}}(Q)}+ \|m\|_{C^{2+\alpha,1+\frac{\alpha}{2}}(Q)}\leq C\|\psi(x)\|_{ C^{2+\alpha}(\Omega)}.
		\end{equation}
		Furthermore, the solution $(u,m)$ is unique within the class
		\begin{equation}
			\{ (u,m)\in  C^{2+\alpha,1+\frac{\alpha}{2}}(Q)\times C^{2+\alpha,1+\frac{\alpha}{2}}(Q): \|(u,m)\|_{ C^{2+\alpha,1+\frac{\alpha}{2}}(Q)}\leq C\delta \}.
		\end{equation}		
		\item[(b)] Define a mapping
		\[
		S: B_{\delta}(C^{2+\alpha}(\Omega))\to C^{2+\alpha,1+\frac{\alpha}{2}}(Q)\times C^{2+\alpha,1+\frac{\alpha}{2}}(Q)\ \mbox{by $S(\psi(x)):=(u,m)$}. 
		\] 
		where $(u,m)$ is the unique solution to the MFG system \eqref{main}.
		Then for any $\psi(x)\in B_{\delta}(C^{2+\alpha}(\Omega))$, $S$ is holomorphic at $\psi$.
	\end{enumerate}
\end{thm}

The proofs of the above two theorems are similar and we only give that of Theorem $\ref{local_wellpose2}$. 

\begin{proof}[Proof of Theorem $\ref{local_wellpose2}$]
		Let 
	\begin{align*}
		&Z_1:= \{m\in C^{2+\alpha}(\Omega ) :\p_{\nu}m=0      \}, \\
		&Z_2:\{(u,m)\in C^{2+\alpha,1+\frac{\alpha}{2}}(Q)\times C^{2+\alpha,1+\frac{\alpha}{2}}(Q):\p_{\nu}m=\p_{\nu}u=0  \text{ in }\Sigma\},\\
		&Z_3:=Z_1\times Z_1\times C^{\alpha,\frac{\alpha}{2}}(Q )\times C^{\alpha,\frac{\alpha}{2}}(Q ),
	\end{align*} 
	and define a map $\mathscr{K}:Z_1\times Z_2 \to Z_3$ by that for any $(m_0(x),\tilde u,\tilde m)\in Y_1\times Y_2$,
	\begin{align*}
		&
		\mathscr{K}( m_0(x),\tilde u,\tilde m)(x,t)\\
		:=&\big( 0, \tilde m(x,0) , 
		-\p_t\tilde u(x,t)-\beta\Delta \tilde u(x,t)\\ &+\frac{\kappa(x)|\nabla \tilde u(x,t)|^2}{2}- F(x,\tilde m(x,t)), 
		\p_t \tilde m(x,t)-\beta\Delta \tilde m(x,t)-{\rm div}(\tilde m(x,t)\kappa(x)\nabla \tilde u(x,t))  \big) .
	\end{align*}
  We can make use of the argument similar to that in the proof of Theorem $\ref{local_wellpose1}$. Since $F\in\mathcal{B}$ and $K(x,y)$ are smooth, it is clear that $\mathscr{K}$ is well-defined. To complete the proof, it suffices to show that
  	\begin{equation}\label{Fer diff'}
  	\begin{aligned}
  		\nabla_{(\tilde u,\tilde m)} \mathscr{K}(0,0,0) (u,m) =(u_{t=T}, m_{t=0}, 
  		-\p_tu-\beta\Delta u-\int_{\Omega}K(x,y)m(y,t)dy, \p_t m-\beta\Delta m),
  	\end{aligned}			
  \end{equation}
is a linear isomorphism between $Z_2$ and $Z_3$, which readily follows from the fact that  $K(x,y)$ is smooth. 

The proof is complete. 
	
\end{proof}

\section{Analysis of the linearized systems in two different scenarios }\label{analysis of lin}
We next develop a high-order linearization/variation scheme of the MFG system \eqref{main}. High-order Linearization method is already a proven skills to study non-linear PDEs. One may refer to \cite{LLLZ} for the parabolic equations case. Since we consider the case that $m\in\mathcal{P}(\Omega)$, one may develop similar Linearization process in $\Omega$ \cite{num_boundary}(see section 5). We give a brief introduction here.

\begin{defi}\label{def_der_1}
	Let $U :\mathcal{P}(\Omega)\to\mathbb{R}$. We say that $U$ is of class $C^1$ if there exists a continuous map $K:  \mathcal{P}(\Omega)\times \Omega\to\mathbb{R}$ such that, for all $m_1,m_2\in\mathcal{P}(\Omega) $ we have
	\begin{equation}\label{derivation}
		\lim\limits_{s\to 0^+}\frac{U\big(m_1+s(m_2-m_1)-U(m_1)\big)}{s}=\int_{\Omega} K(m_1,x)d(m_2-m_1)(x).
	\end{equation}
\end{defi}
Note that the definition of $K$ is up to additive constants. We define  the derivative
$\dfrac{\delta U}{\delta m}$ as the unique map $K$ satisfying $\eqref{derivation}$ and 
\begin{equation}
	\int_{\Omega} K(m,x) dm(x)=0.
\end{equation}
Similarly, we can define higher order derivatives of $U$, and we refer to \cite{high_order_der} for more related discussion.

Notice that if $F\in\mathcal{A}$ , then $F$ depends on the distribution $m$ locally. We have that 
$$
\dfrac{\delta F}{ \delta m}(x,0)(\rho(x,t)):=\left<\dfrac{\delta F}{ \delta m}(x,0,\cdot),\rho(x,t)\right>_{L^2}=
F^{(1)}(x)\rho(x,t),
$$ 
up to a constant. If $F\in\mathcal{B}$, we have
$$
\dfrac{\delta F}{ \delta m}(x,0)(\rho(x,t)):=\left<\dfrac{\delta F}{ \delta m}(x,0,\cdot),\rho(x,t)\right>_{L^2}=
\int_{\Omega}K(x,y)\rho(y,t)dy.
$$ 
up to a constant. 

We consider this linearization process for input $\psi$ and $m_0$. If we fixed $m_0$, then it is just classical linearization process. When one fix $\psi$, we have to work in $\mathcal{P}(\Omega)$ and we use the definition above.
 
\subsection{Higher-order linearization}\label{HLM}
We introduce the basic setting of this higher order
linearization method. Consider the system $\eqref{main}$. Let 
$$\psi(x;\varepsilon)=\varepsilon_1f_1+\varepsilon_2f_2,$$
where 
\[
f_l\in C^{2+\alpha}(\mathbb{R}^n),
\]
and $\varepsilon=(\varepsilon_1,\varepsilon_2)\in\mathbb{R}^2$ with 
$|\varepsilon|=|\varepsilon_1|+|\varepsilon_2|$ small enough. 

By Theorem $\ref{local_wellpose1'}$, there exists a unique solution $(u(x,t;\varepsilon),m(x,t;\varepsilon) )$ of $\eqref{main}$. If $\varepsilon=0$ and $m_0(x)=0$ by our assumption, we have $(u(x,t;0),m(x,t;0) )= (0,0)$. Theorem $\ref{local_wellpose1'}$ also implies that $u(x,t;\varepsilon),m(x,t;\varepsilon)$ are infinitely many differentiable with respect to $\psi.$ Let $S$ be the solution operator of $ \eqref{main}$ with respect to $m_0=0$. Then there exists a bounded linear operator $A$ form $C^{2+\alpha}(\Omega)$ to $[C^{1+\frac{\alpha}{2},2+\alpha}(Q)]^2$ such that
\begin{equation}
	\lim\limits_{\|\psi\|_{C^{2+\alpha}(\Omega)}\to0}\frac{\|S(\psi)-S(0)- A(\psi)\|_{[C^{1+\frac{\alpha}{2},2+\alpha}(Q)]^2}}{\|\psi\|_{C^{2+\alpha(\Omega)}}}=0.
\end{equation} 
If we consider $\varepsilon_2=0$ and fix $f_1$, then it is easy to check that $A(\psi)\Big|_{\varepsilon_1=0}$ is the solution map of the following system which is called the first-order linearization system:
\begin{equation}\label{linear l=1,eg}
	\begin{cases}
	-\p_t u^{(1)}-\beta\Delta u^{(1)}= \int_{\Omega}K(x,y)m^{(1)}(y,t)dy ,  & \text{ in }  Q,\\
	\p_t m^{(1)}(x,t)-\beta\Delta m^{(1)}(x,t)=0, & \text{ in }  Q,\\
	\p_{\nu} u^{(1)}(x,t)=\p_{\nu}m^{(1)}(x,t)=0      &\text{ on } \Sigma,\medskip\\
	u^{(1)}(x,T)=f_1, & \text{ in } \Omega,\medskip\\
	m^{(1)}(x,0)=0, & \text{ in } \Omega.\\
\end{cases}
\end{equation}
In the following, we define
\begin{equation}\label{eq:ld1}
 (u^{(1)}, m^{(1)} ):=A(\psi)\Big|_{\varepsilon_1=0}. 
 \end{equation}
For notational convenience, we write
\begin{equation}\label{eq:ld2}
u^{(1)}=\p_{\varepsilon_1}u(x,t;\varepsilon)|_{\varepsilon=0}\quad\mbox{and}\quad m^{(1)}=\p_{\varepsilon_1}m(x,t;\varepsilon)|_{\varepsilon=0}.
\end{equation}
We shall utilise such notations in our subsequent discussion in order to simplify the exposition and their meaning should be clear from the context. In a similar manner, we can define $u^{(2)}:=\p_{\varepsilon_2}u|_{\varepsilon=0},$
and $m^{(2)}:=\p_{\varepsilon_2}m|_{\varepsilon=0},$
and obtain a similar linearised system.  

Next, we consider 
\begin{equation}\label{eq:ht1}
	u^{(1,2)}:=\p_{\varepsilon_1}\p_{\varepsilon_2}u|_{\varepsilon=0},
	m^{(1,2)}:=\p_{\varepsilon_1}\p_{\varepsilon_2}m|_{\varepsilon=0}.
\end{equation}
Similarly, $( u^{(1,2)},m^{(1,2)})$ can be viewed as the output of the second-order Fr\'echet derivative of $S$ at a specific point. By following similar calculations in deriving \eqref{linear l=1,eg}, one can show that the second-order linearization is given as follows:
\begin{equation}\label{linear l=1,2 eg}
	\begin{cases}
		-\p_tu^{(1,2)}-\beta\Delta u^{(1,2)}(x,t)+\kappa(x)\nabla u^{(1)}\cdot \nabla u^{(2)}\medskip\\
		\hspace*{3cm}= \int_{\Omega}K(x,y)m^{(1,2)}(y,t)dy,& \text{ in } \Omega\times(0,T),\medskip\\
		\p_t m^{(1,2)}-\beta\Delta m^{(1,2)}= {\rm div} (m^{(1)}\kappa(x)\nabla u^{(2)})+{\rm div}(m^{(2)}\kappa(x)\nabla u^{(1)}) ,&\text{ in } \Omega\times (0,T),\medskip\\
		\p_{\nu} u^{(1,2)}(x,t)=\p_{\nu}m^{(1,2)}(x,t)=0      &\text{ on } \Sigma,\medskip\\
		u^{(1,2)}(x,T)=0,& \text{ in } \Omega,\medskip\\
		m^{(1,2)}(x,0)=0, & \text{ in } \Omega.\\
	\end{cases}  	
\end{equation}
Notice that the non-linear terms of the system $\eqref{linear l=1,2 eg}$ depend on the first-order linearized system $\eqref{linear l=1,eg}$. We will utilize this relationship to determine $\kappa(x)$ in the proof of Theorem $\ref{der F2}$. Through this method, we can establish three systems, with the initial values of the first two systems, namely $f_1$ and $f_2$, being arbitrary functions.  Hence, we will focus on the solutions of the first two systems.

By a similar deduction, we can define the high-order linearization scheme of the MFG system $\eqref{main}$ in the case that $F(x,m)\in\mathcal{A}$. For the subsequent use, we only show the main equations in what follows.

The first-order linearization system is given by
\begin{equation}\label{linear l=1,eg'}
	\begin{cases}
		-\p_t u^{(1)}-\beta\Delta u^{(1)}= F^{(1)}(x)m^{(1)}(x,t) ,  & \text{ in }  Q,\\
		\p_t m^{(1)}(x,t)-\beta\Delta m^{(1)}(x,t)=0, & \text{ in }  Q,\\
		\p_{\nu} u^{(1)}(x,t)=\p_{\nu}m^{(1)}(x,t)=0      &\text{ on } \Sigma,\medskip\\
		u^{(1)}(x,T)=f_1, & \text{ in } \Omega,\medskip\\
		m^{(1)}(x,0)=0, & \text{ in } \Omega.\\
	\end{cases}
\end{equation}
The second-order linearization system is given by
\begin{equation}\label{linear l=1,2 eg'}
	\begin{cases}
		-\p_tu^{(1,2)}-\beta\Delta u^{(1,2)}(x,t)+\kappa(x)\nabla u^{(1)}\cdot \nabla u^{(2)}\medskip\\
		\hspace*{3cm}= F^{(1)}m^{(1,2)}(x,t)+F^{(2)}m^{(1)}(x,t)m^{(2)}(x,t),& \text{ in } \Omega\times(0,T),\medskip\\
		\p_t m^{(1,2)}-\beta\Delta m^{(1,2)}= {\rm div} (m^{(1)}\kappa(x)\nabla u^{(2)})+{\rm div}(m^{(2)}\kappa(x)\nabla u^{(1)}) ,&\text{ in } \Omega\times (0,T),\medskip\\
		\p_{\nu} u^{(1,2)}(x,t)=\p_{\nu}m^{(1,2)}(x,t)=0      &\text{ on } \Sigma,\medskip\\
		u^{(1,2)}(x,T)=0,& \text{ in } \Omega,\medskip\\
		m^{(1,2)}(x,0)=0, & \text{ in } \Omega.\\
	\end{cases}  	
\end{equation}
By following a similar manner in defining the first-order variation $u^{(1)}$ and $m^{(1)}$ in \eqref{eq:ld1}--\eqref{eq:ld2}, and for any $N\in\mathbb{N}$, we can consider the high-order variations: 
\begin{equation*}
	u^{(1,2...,N)}=\p_{\varepsilon_1}\p_{\varepsilon_2}...\p_{\varepsilon_N}u|_{\varepsilon=0},
\end{equation*}
\begin{equation*}
	m^{(1,2...,N)}=\p_{\varepsilon_1}\p_{\varepsilon_2}...\p_{\varepsilon_N}m|_{\varepsilon=0}.
\end{equation*}
By following this approach, we can generate a sequence of parabolic systems that will again aid in calculating the higher-order Taylor coefficients of the unknown function $F$. A crucial aspect is that the non-linear terms in higher-order systems rely solely on the solutions of lower-order terms. This allows us to employ mathematical induction in the proof of Theorem $\\ref{der F}$.

\subsection{ High-order variation }\label{sect:hvm}
We have further refined a linearization technique to handle the positivity constraint on the agent population density $m$. To the best of our knowledge, we are the pioneers in proposing and implementing this method to address an inverse problem. This approach not only effectively addresses the positivity constraint in inverse problems but also proves beneficial in scenarios where the running cost of the MFG system is non-locally dependent on $m$.

Consider the system $\eqref{main}$. Let 
$$m_0(x;\varepsilon)=\varepsilon g_1+\varepsilon^2 g_2,$$
where 
\[
g_1, g_2\in C^{2+\alpha}(\mathbb{R}^n)\quad, \quad g_1\geq 0,
\]
and $\varepsilon\in\mathbb{R}^+$. 
Then we have $m_0\geq 0$ in $\Omega$ if $\varepsilon$ tends to $0$.
By Theorem $\ref{local_wellpose2}$, there exists a unique solution $(u(x,t;\varepsilon),m(x,t;\varepsilon) )$ of $\eqref{main}$. If $\varepsilon=0$ and $\psi(x)=0$ , we have $(u(x,t;0),m(x,t;0) )= (0,0).$ 

Let
$$u^{(I)}:=\p_{\varepsilon}u|_{\varepsilon=0},
m^{(I)}:=\p_{\varepsilon}m|_{\varepsilon=0}.$$ 

By direct computations, we can obtain the first-order linearization system as follows:
\begin{equation}
		\begin{cases}
			-\p_t u^{(I)}-\beta\Delta u^{(I)}= \int_{\Omega}K(x,y)m^{(I)}(y,t)dy ,  & \text{ in }  Q,\\
			\p_t m^{(I)}(x,t)-\beta\Delta m^{(I)}(x,t)=0, & \text{ in }  Q,\\
			\p_{\nu} u^{(I)}(x,t)=\p_{\nu}m^{(I)}(x,t)=0      &\text{ on } \Sigma,\medskip\\
			u^{(I)}(x,T)=0, & \text{ in } \Omega,\medskip\\
			m^{(I)}(x,0)=g_1, & \text{ in } \Omega.\\
		\end{cases}	
\end{equation}

In fact, we simply compute the linearization system 5.1 as referenced in \cite{num_boundary} in our specific scenario. So far, this method mirrors the first one. The significant difference lies in our focus on the second-order linearization system instead of the first-order system. In this context, the flexibility of $g_2$ plays a crucial role in tackling the difficulties posed by the positivity constraint in the inverse problem.

Let
$$u^{(II)}:=\p^2_{\varepsilon}u|_{\varepsilon=0},\quad m^{(II)}:=\p^2_{\varepsilon}m|_{\varepsilon=0}.$$
One can obtain that the second-order linearization system:
\begin{equation}
	\begin{cases}
		-\p_t u^{(II)}-\beta\Delta u^{(II)}+\kappa(x)|\nabla u^{(I)}|^2= \int_{\Omega}K(x,y)m^{(II)}(y,t)dy ,  & \text{ in }  Q,\\
		\p_t m^{(II)}(x,t)-\beta\Delta m^{(II)}(x,t)=2\div(m^{(I)}\kappa\nabla u^{(I)}), & \text{ in }  Q,\\
		\p_{\nu} u^{(II)}(x,t)=\p_{\nu}m^{(II)}(x,t)=0      &\text{ on } \Sigma,\medskip\\
		u^{(II)}(x,T)=0, & \text{ in } \Omega,\medskip\\
		m^{(II)}(x,0)=2g_2, & \text{ in } \Omega.\\
	\end{cases}	
\end{equation}

Similar to the observation in Section~\ref{HLM} for the first method, the nonlinear terms depend only on the solution of the first order linearization system. Notice that we only obtain two systems by this method, but the initial value of $m_j^{(II)}(x,0)$ is freely chosen as an arbitrary input. We shall use this fact to recover $K(x,y)$ in the proof of Theorem $\ref{der F2}$.

\section{Proofs of Theorems ~\ref{der F} and ~\ref{der F2} }
\subsection{The main proofs}
We begin with a critical auxiliary lemma.  
\begin{lem}\label{E-F is complete}
	Consider
	\begin{equation}\label{eigenfunction}
		\begin{cases}
			-\p_t u-\beta\Delta u=0  &\text{ in } Q\\
			\p_{\nu} u(x,t)=0   &\text{ in } \Sigma\\
		\end{cases}
	\end{equation}
	There exist a sequence of solution $u(x,t)$ of system $\eqref{eigenfunction}$ such that
	\begin{enumerate}
	\item[(1)] $u(x,t)=e^{\lambda t}g(x;\lambda)$ for some $\lambda\in\mathbb{R}^n$ and $g(x;\lambda)\in C^2(\Omega)$;
	\item[(2)] There does not exit an open subset $U$ of $\Omega$ such that $\nabla g(x;\lambda)= 0$ in $U$. 
	\end{enumerate}
\end{lem}
\begin{proof}
	Let $\dfrac{\lambda}{\beta}$  be an Neumann Laplacian eigenvalue and $g(x;\lambda)$ be a corresponding eigenfunction; that is
	\begin{equation}
		\begin{cases}
			-\Delta g(x;\lambda)=\dfrac{\lambda}{\beta} g(x;\lambda) &\text{ in } \Omega\\
			\p_{\nu} g(x;\lambda)=0     &\text{ in } \Sigma.
		\end{cases}
	\end{equation}
	Then it is obvious that $u(x,t)=e^{\lambda t}g(x;\lambda) $ is a solution of $\eqref{eigenfunction}$. Furthermore, if we suppose there is an open subset $U$ of $\Omega$ such that $\nabla g=0$ in $U$, then $g$ in a constant in $U$. This implies that $\lambda=0.$ It is a contradiction.
	
 The proof is complete. 
\end{proof}

We are now in a position to present the proofs of Theorems~\ref{der F} and \ref{der F2} .
 
\begin{proof}[ Proof of Theorem $\ref{der F}$ ]
		For $j=1,2$, let us consider 
	\begin{equation}\label{MFG 1,2'}
		\begin{cases}
			-u_t-\beta\Delta u+\frac{1}{2}\kappa_j|\nabla  u|^2= F_j(x,m) & \text{ in } Q,\medskip\\
			m_t-\beta\Delta m-\div (m\kappa_j\nabla u)=0  & \text{ in } Q, \medskip\\
			\p_{\nu}u(x,t)=\p_{\nu}m(x,t)=0     & \text{ on } \Sigma, \medskip\\
			u(x,T)=\psi(x)     & \text{ in } \Omega,\medskip\\
			m(x,0)=m_0(x) & \text{ in } \Omega.\\
		\end{cases}
	\end{equation}
	Next, we divide our proof into four steps. 
	
	\medskip
	\noindent {\bf Step I.} We first show $\kappa_1=\kappa_2$. Let $\psi(x)=\varepsilon_1f_1+\varepsilon_2f_2$, $m_0= 0.$ 
	Set
	$$u^{(1)}:=\p_{\varepsilon_1}u|_{\varepsilon=0}=\lim\limits_{\varepsilon\to 0}\frac{u(x,t;\varepsilon)-u(x,t;0) }{\varepsilon_1},$$
	$$m^{(1)}:=\p_{\varepsilon_1}m|_{\varepsilon=0}=\lim\limits_{\varepsilon\to 0}\frac{m(x,t;\varepsilon)-m(x,t;0) }{\varepsilon_1}.$$
	Then we have 
	\begin{equation}\label{first-order-analytic}
		\begin{cases}
			-\p_t u_j^{(1)}-\beta\Delta u_j^{(1)}= F_j^{(1)}(x)m^{(1)}(x,t) ,  & \text{ in }  Q,\\
			\p_t m_j^{(1)}(x,t)-\beta\Delta m_j^{(1)}(x,t)=0, & \text{ in }  Q,\\
			\p_{\nu} u_j^{(1)}(x,t)=\p_{\nu}m_j^{(1)}(x,t)=0      &\text{ on } \Sigma,\medskip\\
			u_j^{(1)}(x,T)=f_1, & \text{ in } \Omega,\medskip\\
			m_j^{(1)}(x,0)=0, & \text{ in } \Omega.\\
		\end{cases}
	\end{equation}
	This implies that $m^{(1)}(x,t)=0.$ We define $m^{(2)}(x,t)$ as the solution for the second order linearisation system (see Section~\ref{analysis of lin}). Similarly, we have $m^{(2)}(x,t)=0.$ 			Therefore, $u_j^{(1)}(x,t)$ are independent of $j$. Let $u_1^{(1)}(x,t)=u_2^{(1)}(x,t)=u^{(1)}(x,t) $, then it satisfies the following system: 
	\begin{equation}\label{first order-analytic'}
		\begin{cases}
			-\p_t u^{(1)}(x,t)-\beta\Delta u^{(1)}(x,t)= 0 ,  & \text{ in }  Q,\\
			\p_{\nu} u^{(1)}(x,t)=0      &\text{ on } \Sigma,\medskip\\
			u^{(1)}(x,T)=f_1, & \text{ in } \Omega,\medskip\\
		\end{cases}
	\end{equation}
Similarly, $u_1^{(2)}(x,t)=u_2^{(2)}(x,t)=u^{(2)}(x,t) $ also satisfy $\eqref{first order-analytic'}$.
	Noticing that $m^{(1)}(x,t)= m^{(2)}(x,t)=0$,
	we have that the second order linearization system is given by (see Section~\ref{analysis of lin})
	\begin{equation}\label{second order-analytic}
		\begin{cases}
			-\p_t u_j^{(1,2)}-\beta\Delta u_j^{(1,2)}+\kappa_j(x)\nabla u^{(1)}\cdot\nabla u^{(2)}= F_j^{(1)}m^{(1,2)}(x,t) ,  & \text{ in }  Q,\\
			\p_t m_j^{(1,2)}(x,t)-\beta\Delta m_j^{(1,2)}(x,t)=0, & \text{ in }  Q,\\
			\p_{\nu} u_j^{(1,2)}(x,t)=\p_{\nu}m_j^{(1,2)}(x,t)=0      &\text{ on } \Sigma,\medskip\\
			u^{(1,2)}(x,T)=0, & \text{ in } \Omega,\medskip\\
			m^{(1,2)}(x,0)=0, & \text{ in } \Omega.\\
		\end{cases}
	\end{equation}
	
	Note that $m^{(1,2)}(x,t)$ must be $0$ and hence, 
	\begin{equation}
		-\p_t u^{(1,2)}-\beta\Delta u^{(1,2)}+\kappa_j(x)\nabla u^{(1)}\cdot\nabla u^{(2)}=0,
	\end{equation}
	holds if $u^{(1)}, u^{(2)}$ are solution of $\eqref{first order}$. Let $\overline{u}(x,t)=u_1^{(1,2)}(x,t)-u_2^{(1,2)}(x,t)$. Since 	$\mathcal{N}_{F_1,H_1}=\mathcal{N}_{F_2,H_2}$,  we have
	\begin{equation}\label{u_1-u_2'}
		\begin{cases}
			-\p_t \overline{u}-\beta\Delta \overline{u}+(\kappa_1(x)-\kappa_2(x))\nabla u^{(1)}\cdot\nabla u^{(2)}=0, & \text{ in }  Q,\\
			\p_{\nu}\overline{u}(x,t)=\overline{u}(x,t)=0&\text{ on } \Sigma,\medskip\\
			\overline{u}(x,T)=\overline{u}(x,0)=0, & \text{ in } \Omega.	
			\end{cases}
	\end{equation}
	Let $\omega$ be a solution of the following system
	\begin{equation}\label{adjoint'}
			\p_t \omega-\beta\Delta \omega=0   \text{ in }  Q.
	\end{equation}
	We multiply $\omega$ on the both sides of $\eqref{u_1-u_2'}$. Then integration by part implies that
	\begin{equation}\label{integral by part 1'}
		\int_{\Omega} (\kappa_1(x)-\kappa_2(x))\nabla u^{(1)}\cdot\nabla u^{(2)}\omega\, dx\,dt =0. 
	\end{equation}
	
	By Lemma $\ref{E-F is complete}$, there exists $\lambda\in\mathbb{R}$ and $g(x)\in C^{\infty}(\Omega)$ such that 
	$e^{\lambda t}g(x)$ satisfies  $\eqref{first order-analytic'}$.
	Let $f_1=e^{\lambda T}g(x)$. Then by the uniqueness of the solution of the heat equation, we have 
	$$  u^{{(1)}}(x,t)=e^{\lambda t}g(x). $$
	Then we have 
	\begin{equation}\label{integral by part 1'' }
		\int_{Q} (\kappa_1(x)-\kappa_2(x))e^{2\lambda t}|\nabla g(x)|^2\omega\, dx\,dt =0. 
	\end{equation}
	Consider  $\omega= e^{-\beta|\xi|^2t-\mathrm{i}x\cdot\xi}$ with $\mathrm{i}:=\sqrt{-1}$ for $\xi\in\mathbb{R}^n.$
	It follows that 
	$$\int_{0}^{T} e^{2\lambda t}e^{-|\xi|^2t }\int_{\Omega}(\kappa_1(x)-\kappa_2(x))|\nabla g(x)|^2 e^{-\mathrm{i}x\cdot\xi } =0,$$
	which readily yields 
	$$\int_{\Omega}(\kappa_1(x)-\kappa_2(x))|\nabla g(x)|^2 e^{-\mathrm{i}x\cdot\xi } =0.$$
	Therefore, we have $ (\kappa_1(x)-\kappa_2(x))|\nabla g(x)|^2=0$ in $\Omega.$ By the construction in Lemma $\ref{E-F is complete}$, we have
	
	$$ \kappa_1(x)-\kappa_2(x)=0. $$
	
	\noindent {\bf Step II.}	Let $\kappa=\kappa_1=\kappa_2$. Next, we derive that $F_1=F_2$. 
	 Consider the following system 	
	 \begin{equation}\label{MFG 1,2-analytic}
	 	\begin{cases}
	 		-u_t-\beta\Delta u+\frac{1}{2}\kappa(x)|\nabla  u|^2= F_j(x,m) & \text{ in } Q,\medskip\\
	 		m_t-\beta\Delta m-\div (m\kappa(x)\nabla u)=0  & \text{ in } Q, \medskip\\
	 		\p_{\nu}u(x,t)=\p_{\nu}m(x,t)=0     & \text{ on } \Sigma, \medskip\\
	 		u(x,T)=\psi(x)     & \text{ in } \Omega,\medskip\\
	 		m(x,0)=m_0(x) & \text{ in } \Omega.\\
	 	\end{cases}
	 \end{equation} 
	  Let $\psi(x)=0$ and 
	$$m_0(x;\varepsilon)=\sum_{l=1}^{N}\varepsilon_lf_l,$$
	where 
	\[
	f_l\in C^{2+\alpha}(\mathbb{R}^n)\quad\mbox{and}\quad f_l\geq 0,
	\]
	and $\varepsilon=(\varepsilon_1,\varepsilon_2,...,\varepsilon_N)\in\mathbb{R}_+^N$ with 
	$|\varepsilon|=\sum_{l=1}^{N}|\varepsilon_l|$ small enough. 
	First, we do the first order linearization to the MFG system \eqref{MFG 1,2-analytic} in $Q$ and can derive: 
	\begin{equation}\label{linearization}
		\begin{cases}
			-\p_{t}u^{(1)}_j-\beta\Delta u_j^{(1)}= F_j^{(1)}(x)m_j^{(1)} & \text{ in } Q, \medskip\\
			\p_{t}m^{(1)} _j-\beta\Delta m_j^{(1)} =0  & \text{ in } Q, \medskip\\
			\p_{\nu}u^{(1)}(x,t)=\p_{\nu}m^{(1)}(x,t)=0     & \text{ on } \Sigma, \medskip\\
			 u^{(1)}(x,T)=0   & \text{ in } \Omega,\medskip\\
			m^{(1)} _j(x,0)=f_1(x) & \text{ in } \Omega.\\
		\end{cases}
	\end{equation}
We choose $f_1(x)=1$. Then we have $m_1^{(1)}(x,t)=m_2^{(1)}(x,t)=1.$ Hence, we have that $u_j^{(1)}(x,t)$ is the solution of the following system:
	\begin{equation}\label{linearization-without-m}
	\begin{cases}
		-\p_{t}u^{(1)}_j-\beta\Delta u_j^{(1)}= F_j^{(1)}(x) & \text{ in } Q, \medskip\\
		\p_{\nu}u^{(1)}(x,t)=0     & \text{ on } \Sigma, \medskip\\
		u^{(1)}(x,T)=0   & \text{ in } \Omega. 
	\end{cases}
\end{equation}
Let $\overline{u}(x,t)=u_1^{(1)}(x,t)-u_2^{(1)}(x,t)$. Since  $\mathcal{N}_{F_1,H_1}=\mathcal{N}_{F_2,H_2} $, we have
\begin{equation}\label{u_1-u_2-analytic}
	\begin{cases}
		-\p_t \overline{u}(x,t)-\beta\Delta \overline{u}(x,t) = F_1^{(1)}(x)-F_2^{(1)}(x), & \text{ in }  Q,\\
		\p_{\nu}\overline{u}(x,t)=\overline{u}(x,t)=0&\text{ on } \Sigma,\medskip\\
		\overline{u}(x,T)=\overline{u}(x,0)=0, & \text{ in } \Omega.	
	\end{cases}
\end{equation}
Let $\omega$ be a solution of the following system
\begin{equation}\label{adjoint-anlaytic}
	\p_t \omega-\beta\Delta \omega=0   \text{ in }  Q. 	
\end{equation}
We multiply $\omega$ on the both side of $\eqref{u_1-u_2-analytic}$. Then integration by part implies that
\begin{equation}
	\int_{Q}  \,(F_1^{(1)}(x)-F_2^{(1)}(x))\omega(x,t)\,dx\,dt=0.
\end{equation}
	Consider  $\omega= e^{-\beta|\xi|^2t-\mathrm{i}x\cdot\xi}$  for $\xi\in\mathbb{R}^n.$
Similar to the argument in Step I, we have
\begin{equation}
	\int_{\Omega}  \,(F_1^{(1)}(x)-F_2^{(1)}(x))e^{\mathrm{i}\xi\cdot x}\,dx=0,
\end{equation}
for all $\xi\in\mathbb{R}^n.$ Hence, we have $F_1^{(1)}(x)=F_2^{(1)}(x)$.

\medskip

\noindent {\bf Step III.}
We proceed to consider the second linearization to the MFG system $\eqref{MFG 1,2-analytic}$ in $Q$ and can obtain for $j=1,2$:
	\begin{equation}
	\begin{cases}
		-\p_tu_j^{(1,2)}-\beta\Delta u_j^{(1,2)}(x,t)+\kappa(x)\nabla u_j^{(1)}\cdot \nabla u_j^{(2)}\medskip\\
		\hspace*{3cm}= F_j^{(1)}(x)m_j^{(1,2)}+F_j^{(2)}(x)m_j^{(1)}m_j^{(2)} & \text{ in } \Omega\times(0,T),\medskip\\
		\p_t m_j^{(1,2)}-\beta\Delta m_j^{(1,2)}= {\rm div} (m_j^{(1)}\kappa(x)\nabla u_j^{(2)})+{\rm div}(m_j^{(2)}\kappa(x)\nabla u_j^{(1)}) ,&\text{ in } \Omega\times (0,T) \medskip\\
		\p_{\nu}u^{(1,2)}(x,t)=\p_{\nu}m^{(1,2)}(x,t)=0     & \text{ on } \Sigma, \medskip\\
		u_j^{(1,2)}(x,T)=0 & \text{ in } \Omega,\medskip\\
		m_j^{(1,2)}(x,0)=0 & \text{ in } \Omega.\\
	\end{cases}  	
\end{equation}
Now we choose $f_1(x)=f_2(x)=1$. Then we have $$m_j^{(1)}(x,t)=m_j^{(2)}(x,t)=1.$$
Notice that we have shown that $F_1^{(1)}(x)=F_2^{(1)}(x)$ in $\Omega$ and this implies that
$$ m^{(1,2)}(x,t)=m_1^{(1,2)}(x,t)=m_2^{(1,2)}(x,t). $$

Let $\hat{u}(x,t)=u_1^{(1,2)}(x,t)-u_2^{(1,2)}(x,t)$. Since  $\mathcal{N}_{\kappa_1,F_1}=\mathcal{N}_{\kappa_2,F_2} $, we have
\begin{equation}\label{u_1-u_2-analytic2}
	\begin{cases}
		-\p_t \hat{u}(x,t)-\beta\Delta \hat{u}(x,t) =( F_1^{(2)}(x)-F_2^{(2)}(x)) , & \text{ in }  Q,\\ 
		\p_{\nu}\hat{u}(x,t)=\hat{u}(x,t)=0, &\text{ on } \Sigma,\medskip\\
		\hat{u}(x,T)=\hat{u}(x,0)=0, & \text{ in } \Omega.	
	\end{cases}
\end{equation}
Let $\omega$ be a solution of the following system
\begin{equation}\label{adjoint-anlaytic2}
	\p_t \omega-\beta\Delta \omega=0   \text{ in }  Q.
\end{equation}
We multiply $\omega$ on the both side of $\eqref{u_1-u_2-analytic2}$. Then integration by part implies that
\begin{equation}
	\int_{Q}  \,(F_1^{(2)}(x)-F_2^{(2)}(x))\omega(x,t)\,dx\,dt=0.
\end{equation}
Similar to the argument in Step II, we can show that
$$F_1^{(2)}(x)-F_2^{(2)}(x)=0.$$
\noindent {\bf Step IV.}
Finally, using mathematical induction and repeating similar arguments in Steps II and III, one can show that
$$F^{(k)}_1(x)-F^{(k)}_2(x)=0 ,$$
for all $k\in\mathbb{N}$. Hence, $$F_1(x,z)=F_2(x,z)\quad \mbox{in}\ \ \Omega\times\mathbb{R}.$$

The proof is complete. 

	\end{proof}

Next we present the proof of Theorem $\ref{der F2}$. Recall that in this case, we have 
$$ F(x,m)= \int_{\Omega}K(x,y)m(y,t)\, dy.$$
\begin{proof}[ Proof of Theorem $\ref{der F2}$ ]
	For $j=1,2$, let us consider 
	\begin{equation}\label{MFG 1,2}
		\begin{cases}
			-u_t-\Delta u+\frac{1}{2}\kappa_j|\nabla  u|^2= \int_{\Omega}K_j(x,y)m(y,t)dy & \text{ in } Q,\medskip\\
			m_t-\Delta m-\div (m\kappa_j\nabla u)=0  & \text{ in } Q, \medskip\\
			\p_{\nu}u(x,t)=\p_{\nu}m(x,t)=0     & \text{ on } \Sigma, \medskip\\
			u(x,T)=\psi(x)     & \text{ in } \Omega,\medskip\\
			m(x,0)=m_0(x) & \text{ in } \Omega.\\
		\end{cases}
	\end{equation}
	Next, we divide our proof into two steps. 
	
	\medskip
	
	\noindent {\bf Step I.} We first prove that $\kappa_1=\kappa_2$.
	
			Let $\psi(x)=\varepsilon_1f_1+\varepsilon_2f_2$, $m_0= 0.$ 
			Set
			$$u^{(1)}:=\p_{\varepsilon_1}u|_{\varepsilon=0}=\lim\limits_{\varepsilon\to 0}\frac{u(x,t;\varepsilon)-u(x,t;0) }{\varepsilon_1},$$
			$$m^{(1)}:=\p_{\varepsilon_1}m|_{\varepsilon=0}=\lim\limits_{\varepsilon\to 0}\frac{m(x,t;\varepsilon)-m(x,t;0) }{\varepsilon_1}.$$
			By straightforward calculations, one can show that 
			\begin{equation}\label{first order}
				\begin{cases}
					-\p_t u_j^{(1)}-\beta\Delta u_j^{(1)}= \int_{\Omega}K_j(x,y)m^{(1)}(y,t)dy ,  & \text{ in }  Q,\\
					\p_t m_j^{(1)}(x,t)-\beta\Delta m_j^{(1)}(x,t)=0, & \text{ in }  Q,\\
					\p_{\nu} u_j^{(1)}(x,t)=\p_{\nu}m_j^{(1)}(x,t)=0      &\text{ on } \Sigma,\medskip\\
					u_j^{(1)}(x,T)=f_1, & \text{ in } \Omega,\medskip\\
					m_j^{(1)}(x,0)=0, & \text{ in } \Omega.\\
				\end{cases}
			\end{equation}
			This implies that $m^{(1)}(x,t)=0.$ We define $m^{(2)}(x,t)$ in the way as defined in Section~\ref{analysis of lin}). In a similar way, we can show that $m^{(2)}(x,t)=0.$ 			Therefore, $u_j^{(1)}(x,t)$ are independent of $j$. Let $u_1^{(1)}(x,t)=u_2^{(1)}(x,t)=u^{(1)}(x,t) $. One can readily see that it satisfies the following system:
			\begin{equation}\label{first order'}
				\begin{cases}
					-\p_t u^{(1)}(x,t)-\beta\Delta u^{(1)}(x,t)= 0 ,  & \text{ in }  Q,\\
					\p_{\nu} u^{(1)}(x,t)=0      &\text{ on } \Sigma,\medskip\\
					u^{(1)}(x,T)=f_1, & \text{ in } \Omega,\medskip\\
				\end{cases}
			\end{equation}
			
			By direct calculations, we have that the second order linearization system is given by
			\begin{equation}\label{second order}
				\begin{cases}
					-\p_t u_j^{(1,2)}-\beta\Delta u_j^{(1,2)}+\kappa_j(x)\nabla u^{(1)}\cdot\nabla u^{(2)}= \int_{\Omega}K_j(x,y)m^{(1,2)}(y,t)dy ,  & \text{ in }  Q,\\
					\p_t m_j^{(1,2)}(x,t)-\beta\Delta m_j^{(1,2)}(x,t)=0, & \text{ in }  Q,\\
					\p_{\nu} u_j^{(1,2)}(x,t)=\p_{\nu}m_j^{(1,2)}(x,t)=0      &\text{ on } \Sigma,\medskip\\
					u^{(1,2)}(x,T)=0, & \text{ in } \Omega,\medskip\\
					m^{(1,2)}(x,0)=0, & \text{ in } \Omega.\\
				\end{cases}
			\end{equation}
			Note that $m^{(1,2)}(x,t)$ must be $0$ and hence it holds that
			\begin{equation}
				-\p_t u^{(1,2)}-\beta\Delta u^{(1,2)}+\kappa_j(x)\nabla u^{(1)}\cdot\nabla u^{(2)}=0,
			\end{equation}
			provided that $u^{(1)}, u^{(2)}$ are solutions to $\eqref{first order}$. Let $\overline{u}=u_1^{(1,2)}(x,t)-u_2^{(1,2)}(x,t)$. Since 	$\mathcal{N}_{F_1,H_1}=\mathcal{N}_{F_2,H_2}$,  we have
			\begin{equation}\label{u_1-u_2}
				\begin{cases}
					-\p_t \overline{u}-\beta\Delta \overline{u}+(\kappa_1(x)-\kappa_2(x))\nabla u^{(1)}\cdot\nabla u^{(2)}=0, & \text{ in }  Q,\\
					\p_{\nu}\overline{u}(x,t)=\overline{u}(x,t)=0&\text{ on } \Sigma,\medskip\\
					\overline{u}(x,T)=0, & \text{ in } \Omega.
				\end{cases}
			\end{equation}
			Let $\omega$ be a solution of the following system
			\begin{equation}\label{adjoint}
					\p_t \omega-\beta\Delta \omega=0,  \text{ in }  Q. 
			\end{equation}
			Multiplying $\omega$ on the both sides of $\eqref{u_1-u_2}$ and then integration by part yield that
			\begin{equation}\label{integral by part 1}
				\int_{\Omega} (\kappa_1(x)-\kappa_2(x))\nabla u^{(1)}\cdot\nabla u^{(2)}\omega\, dx\,dt =0. 
			\end{equation}
			
			By Lemma $\ref{E-F is complete}$, there exists $\lambda\in\mathbb{R}$ and $g(x)\in C^{\infty}(\Omega)$ such that 
			$e^{\lambda t}g(x)$ satisfies  $\eqref{first order'}$.
			Let $f_1=e^{\lambda T}g(x)$. Then by the uniqueness of the solution of the heat equation, we have 
			$$  u^{{(1)}}(x,t)=e^{\lambda t}g(x). $$

		Consider  $\omega= e^{-\beta|\xi|^2t-\mathrm{i}x\cdot\xi}$  for $\xi\in\mathbb{R}^n.$
		By a similar argument to that in the proof of Theorem $\ref{der F}$, we have
			$$\int_{0}^{T} e^{2\lambda t}e^{-\beta|\xi|^2t }\int_{\Omega}(\kappa_1(x)-\kappa_2(x))|\nabla g(x)|^2 e^{-\mathrm{i}x\cdot\xi } =0,$$
			i.e.
			$$\int_{\Omega}(\kappa_1(x)-\kappa_2(x))|\nabla g(x)|^2 e^{-\mathrm{i}x\cdot\xi } =0.$$
			Therefore, we have $ (\kappa_1(x)-\kappa_2(x))|\nabla g(x)|^2=0$ in $\Omega.$ By the construction in Lemma $\ref{E-F is complete}$, we have
			
			$$ \kappa_1(x)-\kappa_2(x)=0. $$
			
			\medskip
			
		\noindent {\bf Step II.}	Let $\kappa=\kappa_1=\kappa_2$. Next, we show that $K_1(x,y)=K_2(x,y)$. To that end, we need to make use of the linearisation method presented in Section~\ref{sect:hvm}. Let 
			$m_0=\varepsilon g_1 +\varepsilon^2 g_2$, where $g_1>0, \varepsilon>0 $. Let $\psi(x)=0.$ Then the first order linearization system is given by
			
			\begin{equation}\label{first order type2}
				\begin{cases}
					-\p_t u_j^{(I)}-\beta\Delta u_j^{(I)}= \int_{\Omega}K_j(x,y)m^{(I)}(y,t)dy ,  & \text{ in }  Q,\\
					\p_t m_j^{(I)}(x,t)-\beta\Delta m_j^{(I)}(x,t)=0, & \text{ in }  Q,\\
					\p_{\nu} u_j^{(I)}(x,t)=\p_{\nu}m_j^{(I)}(x,t)=0      &\text{ on } \Sigma,\medskip\\
					u_j^{(I)}(x,T)=0, & \text{ in } \Omega,\medskip\\
					m_j^{(I)}(x,0)=g_1, & \text{ in } \Omega.\\
				\end{cases}	
			\end{equation}
			
			Let $g_1(x)=1$. Since $\int_{\Omega} K_j(x,y) dy=0$ for $j=1,2$, we have $u_j^{(I)}(x,t)=0, m_j^{(I)}(x,t)=1.$ Then the second order linearization system is given by
			\begin{equation}\label{second order type2}
				\begin{cases}
					-\p_t u_j^{(II)}-\beta\Delta u_j^{(II)}+\kappa(x)|\nabla u^{(I)}|^2= \int_{\Omega}K_j(x,y)m^{(II)}(y,t)dy ,  & \text{ in }  Q,\\
					\p_t m_j^{(II)}(x,t)-\beta\Delta m_j^{(II)}(x,t)=2\div(m^{(I)}\kappa\nabla u^{(I)}), & \text{ in }  Q,\\
					\p_{\nu} u_j^{(II)}(x,t)=\p_{\nu}m_j^{(II)}(x,t)=0      &\text{ on } \Sigma,\medskip\\
					u_j^{(II)}(x,T)=0, & \text{ in } \Omega,\medskip\\
					m_j^{(II)}(x,0)=2g_2, & \text{ in } \Omega.\\
				\end{cases}	
			\end{equation}
			Define $\hat{K}=K_1-K_2$. Let $\omega$ be a solution of $\eqref{adjoint}$. By a similar argument to that in Step~I, we can derive 
			\begin{equation}
				\int_{Q}\big[ \int_{\Omega}  \hat{K}(x,y) m^{(2)}(y,t)  dy \big]\omega(x,t)\,dx\,dt=0,
			\end{equation}
			for all $ m^{(II)}(y,t) $ being solutions of
			\begin{equation}
				\begin{cases}
					\p_t m_j^{(II)}(x,t)-\beta\Delta m_j^{(II)}(x,t)=0, & \text{ in }  Q,\\
					\p_{\nu}m_j^{(II)}(x,t)=0      &\text{ on } \Sigma,\medskip\\
					m_j^{(II)}(x,0)=2g_2, & \text{ in } \Omega.\\
				\end{cases}	
			\end{equation}
			By Lemma $\ref{E-F is complete}$, we choose $m^{(II)}(x,t)=e^{\lambda t}g(x;\lambda)$. Then by a similar argument as above, we can deduce that
			\begin{equation}\label{eq:aa1}
				\int_{0}^{T}e^{\lambda t}e^{-|\xi|^2t}dt\int_{\Omega}\big[ \int_{\Omega}  \hat{K}(x,y) g(y;\lambda)  dy \big] e^{-\mathrm{i}x\cdot\xi}\,dx=0. 
			\end{equation}
			From \eqref{eq:aa1}, we clearly have that
			\begin{equation}\label{eq:aa2} 
			\int_{\Omega}  \hat{K}(x,y) g(y;\lambda)  dy=0\quad \mbox{for all $g(y;\lambda)$}.
			\end{equation}
			Note that $g(y;\lambda)$ can be any Neumann eigenfunction of $-\Delta$ in $\Omega$. Since the Neumann-Laplacian eigenfunctions form a complete set in $L^2(\Omega)$, we readily derive from \eqref{eq:aa2} that
			$$K_1(x,y)=K_2(x,y)\ \ \mbox{in}\ \Omega\times\Omega.$$
			
			 The proof is complete.
		\end{proof}
		
	\begin{rmk}
	We would like to point out that by a closer inspection of the proofs of Theorems~\ref{der F} and \ref{der F2}, it can easily inferred that in order to recover $F$ and $H$, one does not need the full knowledge of $\mathcal{N}_{F,H}(m_0,\psi)$. In fact, one can only use $\mathcal{N}_{F,H}(m_0,0)$ and $\mathcal{N}_{F,H}(0,\psi)$ to recover $F$ and $\kappa(x), respectively.$
\end{rmk}

\subsection{Comparison and discussion of the two linearisation methods}

In order to tackle the MFG inverse problems, we propose two linearisation strategies in Section~\ref{analysis of lin}, which are respectively termed as the high-order linearisation (cf. Section~\ref{HLM}) and high-order variation (cf. Section~\ref{sect:hvm}). In this section, we would like to give a succinct comparison of the two methods. Our point is that the high-order variation method is superior to the high-order linearisation method when addressing the MFG inverse problem with the running cost $F$ depending on the probability measure $m$ non-locally.

In fact, if the running cost $F_j (j=1,2)$ depends on the $m$ locally, both methods works in the proofs of Theorems~\ref{der F} and \ref{der F2}, and we leave the verification to readers. However, if we consider the case that $F_j$ depends on $m$ non-locally, we next show that the high-order linearisation method in general does not work to recover $F$.
First, it is noted that we can always recover $\kappa$ by the high-order linearisation argument. Hence, we can assume that $\kappa_1=\kappa_2$ and consider the linearzation systems obtained by the  high-order linearization method in the non-local case:
\begin{equation}\label{com-1}
		\begin{cases}
		-\p_t u^{(1)}-\beta\Delta u^{(1)}= \int_{\Omega}K_j(x,y)m^{(1)}(y,t)dy ,  & \text{ in }  Q,\\
		\p_t m^{(1)}(x,t)-\beta\Delta m^{(1)}(x,t)=0, & \text{ in }  Q,\\
		\p_{\nu} u^{(1)}(x,t)=\p_{\nu}m^{(1)}(x,t)=0      &\text{ on } \Sigma,\medskip\\
		u^{(1)}(x,T)=0, & \text{ in } \Omega,\medskip\\
		m^{(1)}(x,0)=g_1, & \text{ in } \Omega,\\
	\end{cases}
\end{equation}
and
\begin{equation}\label{com-2}
		\begin{cases}
		-\p_tu^{(1,2)}-\beta\Delta u^{(1,2)}+\kappa(x)\nabla u^{(1)}\cdot \nabla u^{(2)}
		= \int_{\Omega}K_j(x,y)m^{(1,2)}(y,t)dy,& \text{ in } \Omega\times(0,T),\medskip\\
		\p_t m^{(1,2)}-\beta\Delta m^{(1,2)}= {\rm div} (m^{(1)}\kappa(x)\nabla u^{(2)})+{\rm div}(m^{(2)}\kappa(x)\nabla u^{(1)}) ,&\text{ in } \Omega\times (0,T),\medskip\\
		\p_{\nu} u^{(1,2)}(x,t)=\p_{\nu}m^{(1,2)}(x,t)=0      &\text{ on } \Sigma,\medskip\\
		u^{(1,2)}(x,T)=0,& \text{ in } \Omega,\medskip\\
		m^{(1,2)}(x,0)=0, & \text{ in } \Omega.\\
	\end{cases}  	
\end{equation}
Similarly, we can derive the equations for $(u^{(i)}(x,t),m^{(i)}(x,t))$ $(i=1,2,....,N)$.

Notice that we cannot control the second order linearzation systems directly because $\eqref{com-2}$ only depends on the solutions of the first order linearzation systems. 
Recall that we have the probability density constraint. Since $m_0=\sum_{l=1}^{N}\varepsilon_ig_i$, $g_i(x)$ are required to be non-negative. Therefore, one cannot recover $K$ by the high-order linearisation method. Even though we can obtain a sequence of first-order linearzation systems by this method, all initial values are required to be non-negative. For comparison, in the high-order variation method, we only obtain one first order linearzation system that is used to fulfill the positivity constraint, but we can recover the unknown from the second order linearzation system which has no positivity constraint. Hence, in the non-local case, the high-order variation method is more advantageous. Indeed, if one makes use of the high-order linearisation method in such a case, one has to utilise both the first and second order linearisation systems in order to recover $F$ (because the information of the integral kernel $K$ is contained in both systems), and one will suffer from the positivity constraint of the first order linearisation system. 

\section*{Acknowledgments}

 The work was supported by the Hong Kong RGC General Research Funds (projects 11311122, 11300821 and 12301420),  the NSFC/RGC Joint Research Fund (project N\_CityU101/21), and the ANR/RGC Joint Research Grant, A\_CityU203/19.

\end{document}